\documentclass{birkjour}

\usepackage{url}

\usepackage{latexsym}
\usepackage{amsfonts}
\usepackage{amssymb}
\usepackage{color}

\newcommand{\N}{\ensuremath{\mathbb{N}}}

\newcommand{\C}{\ensuremath{\mathbb{C}}}

\def\Re{\emph{\textrm{Re}}}

\newcommand{\qbinomial}[3]{\mbox{$
\biggl[\!\!\!
\begin{array}{c}
#1\\
 #2
\end{array}\!\!\!\biggr]_{
\!{#3}} $} }

\newcommand{\qhypergeom}[5]{\mbox{$
_#1 \phi_#2\left. \left(\!\!\!\!
\begin{array}{c}
\multicolumn{1}{c}{\begin{array}{c} #3
\end{array}}\\[1mm]
\multicolumn{1}{c}{\begin{array}{c} #4
           \end{array}}\end{array}
\right| \displaystyle{#5}\right) $} }

\newcommand{\Cos}{{\textup{{Cos}}}}
\newcommand{\Sin}{\textup{{{Sin}}}}
\newcommand{\Cosh}{{\textup{{Cosh}}}}
\newcommand{\Sinh}{\textup{{{Sinh}}}}

\newcommand{\qcos}{\cos_{q}}
\newcommand{\qsin}{\sin_{q}}
\newcommand{\qCos}{\Cos_{q}}
\newcommand{\qSin}{\Sin_{q}}

\newcommand{\qcosh}{\cosh_{q}}
\newcommand{\qsinh}{\sinh_{q}}
\newcommand{\qCosh}{\Cosh_{q}}
\newcommand{\qSinh}{\Sinh_{q}}

\newcommand{\qlap}[1]{\mathcal{L}^{(#1)}_{2,q}}
\newcommand{\qint}{\displaystyle{\int_{0}^{\infty}}}
\newcommand{\dqint}{\displaystyle{\int_{0}^{\infty}\!\!\!\!\!\int_{0}^{\infty}}}

 \newtheorem{theorem}{Theorem}[section]

 \newtheorem{proposition}[theorem]{Proposition}
 \theoremstyle{definition}
 
 \theoremstyle{remark}
 \newtheorem{remark}[theorem]{Remark}
 
 \numberwithin{equation}{section}

\begin{document}

%
%
%
%
%
%
%
%
%

\title[]
 {On double $q$-Laplace transform and applications}

\author[P. Njionou Sadiang]{P. Njionou Sadjang}

\address{%
University of Douala,\\
Faculty of Industrial Engineering\\
Douala\\
Cameroon}

\email{pnjionou@yahoo.fr}

\thanks{
}



\date{\today}

\begin{abstract}
We introduce four $q$-analogs of the double Laplace transform and prove some of their main properties. Next we show how they can be used to solve some $q$-functional equations and partial $q$-differential equations.
\end{abstract}

\maketitle

\section{Introduction}
\noindent The classical Laplace transform of a function $f$ is given by 
\begin{equation}\label{laplace1}
\mathcal{L}\{f(t)\}(s)=\int_{0}^{\infty}e^{-st}f(t)dt,\quad \quad s=a+ib\in\mathbb{C},\end{equation}
and plays a fundamental role in pure and applied analysis. Laplace transform has been studied very extensively and has found to have a wide variety of applications in mathematical, physical, statistical, and engineering sciences and also in other sciences. There is a very extensive literature available of the Laplace transform of a function $f(t)$ of one variable $t$  and its applications (see for example Churchill \cite{churchill}, Schiff \cite{schiff}, Debnath and Bhatta \cite{debnath} and the references therein). 

The double Laplace transform of a function $f(x,y)$ of two variable was first introduced in 1939 by Berstein in his dissertation \cite{berstein1} (later pubished as an article \cite{berstein2}) as 
\begin{equation}\label{dlt}
\mathcal{L}_2(f(x,y))(r,s)=\int_{0}^{\infty}\!\!\!\!\int_{0}^{\infty}f(x,y)e^{-(rx+sy)}dxdy,
\end{equation}
where $x$ and $y$ are two positive numbers, $r$ and $s$ are complex numbers. Very recently, several interesting properties and applications of the double Laplace transform to functional, integral and partial differential equations have been studied in \cite{debnath1}. 

The development of $q$-analysis started in the 1740s, when Euler initiated the theory of partitions, also called additive analytic number theory. Euler always wrote in Latin and his collected works were published only at the beginning of the 1800s, under the legendary Jacobi. In 1829 Jacobi presented his triple product identity (sometimes called the Gau$\beta$-Jacobi triple product identity), and his $\theta$ and elliptic functions, which in principle are equivalent to $q$-analysis. The progress of $q$-calculus continued under C. F. Gau$\beta$ (1777-1855), who in 1812 invented the hypergeometric series and their contiguity relations. Gau$\beta$ would later invent the $q$-binomial coefficients and prove an identity for them, which forms the basis for $q$-analysis.

The theory of $q$-analysis have been applied in recent past  in many areas of mathematics and physics like ordinary fractional calculus, optimal control problems, quantum calculus, $q$-transform analysis and in finding solutions of the $q$-difference and $q$-integral equations. In 1910, Jackson \cite{jackson} presented a precise definition of the so-called the $q$-Jackson integral and developed $q$-calculus in a systematic way.

In order to deal with $q$-difference equations, $q$-versions of the classical Laplace transform have been consecutively introduced in the literature. Studies of $q$-versions of Laplace transform go back to Hahn \cite{hahn1949}. Abdi \cite{abdi1960,abdi1962,abdi1964} published also many results in this domain. In a recent paper \cite{chung} two very interesting versions of $q$-Laplace transform are introduced as follows
\begin{equation}\label{q1}
L_q(f(t))(s)=\qint E_q(-qst) f(t)d_qt,\quad (s>0). 
\end{equation}
for the first kind and 
\begin{equation}\label{q2}
\mathcal{L}_q(f(t))(s)=\qint e_q(-st) f(t)d_qt,\quad (s>0). 
\end{equation}
for the second kind. Note that both \eqref{q1} and \eqref{q2} generalize \eqref{laplace1}. We will frequently use some properties of \eqref{q1} and \eqref{q2} and will refer the reader to the paper \cite{chung} for more details. 

\noindent In this paper, we introduce four kinds of double $q$-Laplace transforms and prove their main properties. Next, applications are done to solve some classical partial $q$-differential equations that appear in the litterature. The double $q$-Laplace transform introduced here are clearly generalization of the one given in \cite{berstein1}. 

\section{Basic definitions and miscellaneous results}

\subsection{$q$-number, $q$-factorial, $q$-binomial, $q$-power, $q$-addition}

For any complex number $a$, the basic or  $q$-number is defined by
\begin{equation*}
[a]_q=\dfrac{1-q^a}{1-q},\quad q\neq 1.
\end{equation*}
For any non negative integer $n$, the $q$-factorial is defined by 
\begin{equation*} [n]_q!=[n]_q[n-1]_q\cdots[1]_q=\prod_{k=1}^n[k]_q,\quad n\in \N,\quad [0]_q!=1,
\end{equation*}
and the $q$-pochhammer is defined as
\begin{equation*}
(a;q)_0=1,\quad (a;q)_n=\prod_{k=0}^{n-1}(1-aq^k),\;\; n\in\N.
\end{equation*}
The limit, $\lim\limits_{n\to\infty}(a;q)_n$ is denoted by $(a;q)_{\infty}$, provided that $|q|<1$. Then,
\begin{equation*}
(a;q)_n=\dfrac{(a;q)_{\infty}}{(aq^n;q)_{\infty}},\quad n\in\N_{0},\;\;\; |q|<1,
\end{equation*}
and for any complex number $\alpha$, this definition can be extended by 
\begin{equation*}
(a;q)_{\alpha}=\dfrac{(a;q)_{\infty}}{(aq^{\alpha};q)_{\infty}},\quad |q|<1,
\end{equation*}
where the principal value of $q^{\alpha}$ is taken.

\noindent  The $q$-binomial coefficients are defined by 
\begin{equation*}
\qbinomial{n}{k}{q}=\dfrac{[n]_q!}{[k]_q![n-k]_q!}=\dfrac{(q;q)_{n}}{(q;q)_k(q;q)_{n-k}},\quad 0\leq k\leq n.
\end{equation*}
It is worth noting that $\qbinomial{n}{k}{q}=\qbinomial{n}{n-k}{q}$. 

\noindent The $q$-power basis is defined by
\[(x\ominus y)_q^n=\left\{
\begin{array}{ll}
(x-y)(x-yq)\dots(x-yq^{n-1}), & n=1,2,\cdots\\
1& n=0.
\end{array}
\right.\]  
In the same line we introduce the following notation 
\[(x\oplus y)_q^n=\left\{
\begin{array}{ll}
(x+y)(x+yq)\dots(x+yq^{n-1}), & n=1,2,\cdots\\
1& n=0.
\end{array}
\right.\]
It is not difficult to proved that (see \cite{njionou2013})
\begin{equation}\label{qpower2}
(x\oplus y)_q^n=\sum_{k=0}^n\qbinomial{n}{k}{q}q^{\binom{n-k}{2}}x^ky^{n-k}.
\end{equation}  

\noindent In \cite{schork}, Schork has studied Ward's ''Calculus of Sequences'' and introduced a $q$-addition $x\oplus_q y$ by 
\begin{equation}
 (x\oplus_q y)^n=\sum_{k=0}^{n}\qbinomial{n}{k}{q}x^ky^{n-k},
\end{equation}
and although this $q$-addition was already known to Jackson, it was generalized later on by Ward and Al-Salam. For more informations about different $q$-additions, see e.g. \cite{ernst}. Similarly the $q$-subtraction can be defined in the same way by \cite{kim}
\begin{equation}\label{qadd}
 (x\ominus_q y)^n=\sum_{k=0}^{n}\qbinomial{n}{k}{q}x^k(-y^{n-k})= (x\oplus_q (-y))^n. 
\end{equation}
Al-Salam introduced in \cite{al-salam1}  the following $q$-coaddition 
\begin{eqnarray}\label{coqadd}
(x\boxplus_q y)^n=\sum_{k=0}^{n}\qbinomial{n}{k}{q}q^{k(k-n)}x^ky^{n-k}.
\end{eqnarray}
We introduce the following $q$-cosubtraction \cite[P. 233]{ernst}
\begin{eqnarray}\label{coqsub}
(x\boxminus_q y)^n=(x\boxplus_q (-y))^n=\sum_{k=0}^{n}\qbinomial{n}{k}{q}q^{k(k-n)}x^k(-y)^{n-k}.
\end{eqnarray}

\subsection{The $q$-derivative and the $q$-integral}

\noindent  The $q$-derivative operator  is defined by  \cite{kac,KLS}
\[D_qf(x)=\dfrac{f(x)-f(qx)}{(1-q)x},\quad x\neq 0,\]
satisfying  the important product rule 
\begin{equation}\label{productrule}
D_q(f(x)g(x))=f(x)D_qg(x)+g(qx)D_qf(x).
\end{equation}
In this sense, note that when we deal with functions $f(x_1,x_2,\ldots,x_n)$ of more than one variable, we denote $D_q f$ by $D_{q,x_i}f$ or $\dfrac{\partial_q}{\partial_q x_i} f$ to make clear that the derivative is taken with respect to the variable $x_i$. For the case of two variables $x$ and $y$ for example, the $q$-partial derivative with respect to $x$ is given by \cite{rajkovic}
\begin{equation}
   D_{q,x}f(x,y)=\dfrac{f(x,y)-f(qx,y)}{(1-q)x},\;\; (x\neq 0)
\end{equation}
and 
\[ D_{q,x}f(x,y)\big|_{x=0}=\lim\limits_{x\to 0} D_{q,x}f(x,y).  \]

\noindent The $q$-integral operator  is defined by \cite{kac,KLS}
\begin{equation}
\int_0^zf(z)d_qt=z(1-q)\sum_{k=0}^\infty
q^kf(zq^k).
\end{equation}
\noindent This definition can be established based on a simple geometric series. 

\noindent Note that for $a<b$ two real numbers, one has 
\[ \int_a^b f(x)d_qx=\int_0^bf(x)d_qx-\int_0^af(x)d_qx\]
and the $q$-integration by part is 
\[\int_a^bf(x)D_qg(x)d_qx=f(b)g(b)-f(a)g(a)-\int_{a}^{b}g(qx)D_qf(x)d_qx.\]
Note that in this $q$-integration by part, $b=\infty$ is allowed as well \cite{kac}.

\subsection{The $q$-hypergeometric, the $q$-exponential and $q$-trigonometric functions}

The basic hypergeometric or $q$-hypergeometric function $_r\phi_s$ is defined by the series
 \[
 \qhypergeom{r}{s}{a_1,\cdots,a_r}{b_1,\cdots,b_s}{q;z}:=\sum_{k=0}^\infty\frac{(a_1,\cdots,a_r;q)_k}{(b_1,\cdots,b_s;q)_k}\left((-1)^k
 q^{\binom{k}{2}}\right)^{1+s-r}\frac{z^k}{(q;q)_k}. 
 \]
  where
 \[(a_1,\cdots,a_r)_k:=(a_1;q)_k\cdots(a_r;q)_k,\]

\noindent The usual exponential function may have two different natural $q$-extensions, denoted by $e_q(z)$ and $E_q(z)$, which are defined, respectively, by
\begin{equation*}
e_q(z):=\qhypergeom{1}{0}{0}{-}{q;(1-q)z}=\sum_{n=0}^{\infty}\frac{z^n}{[n]_q!},\quad 0<|q|<1,\;\; |z|<1, \label{small-qexp}
\end{equation*}
and
\begin{equation*}
E_q(z):=\qhypergeom{0}{0}{-}{-}{q,-(1-q)z}=\sum_{n=0}^{\infty}\frac{q^{\binom{n}{2}}}{[n]_{q}!}z^n,\quad 0<|q|<1.  \label{big-qexp}
\end{equation*}
\noindent It is worth noting that $e_q(z)$ and $E_q(z)$ are linked by the well known relation
\begin{equation*}\label{eq00}
e_q(z)E_q(-z)=1.
\end{equation*}
They fulfil the $q$-defivative rules 
\begin{eqnarray*}
 D_q e_q(\lambda x)&=&\lambda e_q(\lambda x),\\
 D_qE_q(\lambda x)&=& \lambda E_q(\lambda qx).
\end{eqnarray*}
\noindent  It is not difficult to see that \cite{al-salam1,ernst,chung}
\begin{equation}\label{qaddition}
e_q(x)e_q(t)=e_q(x\oplus_q y),\quad \forall x,y\in\C,
\end{equation}
and 
\begin{equation}\label{qaddition1}
E_q(x)E_q(t)=E_q(x\boxplus_q y),\quad \forall x,y\in\C.
\end{equation}

\noindent From these definitions of the $q$-exponential functions, we derive the following $q$-trigonometric functions \cite{chung,kac}
\begin{eqnarray*}
   \qcos(z)&=&\dfrac{e_q(iz)+e_q(-iz)}{2}= \sum_{n=0}^{\infty}\dfrac{(-1)^nz^{2n}}{[2n]_q!}\\
   \qsin(z)&=&\dfrac{e_q(iz)-e_q(-iz)}{2i}= \sum_{n=0}^{\infty}\dfrac{(-1)^nz^{2n+1}}{[2n+1]_q!}\\ 
   \qCos(z)&=&\dfrac{E_q(iz)+E_q(-iz)}{2}= \sum_{n=0}^{\infty}\dfrac{(-1)^nq^{\binom{2n}{2}}}{[2n]_q!}z^{2n}\\
   \qSin(z)&=& \dfrac{E_q(iz)-E_q(-iz)}{2i}=\sum_{n=0}^{\infty}\dfrac{(-1)^nq^{\binom{2n+1}{2}}}{[2n+1]_q!}z^{2n+1}
\end{eqnarray*}
and the hyperbolic $q$-trigonometric functions 
\begin{eqnarray*}
   \qcosh(z)&=&\dfrac{e_q(z)+e_q(-z)}{2}= \sum_{n=0}^{\infty}\dfrac{z^{2n}}{[2n]_q!}\\
   \qsinh(z)&=&\dfrac{e_q(z)-e_q(-z)}{2}= \sum_{n=0}^{\infty}\dfrac{z^{2n+1}}{[2n+1]_q!}\\ 
   \qCosh(z)&=&\dfrac{E_q(z)+E_q(-z)}{2}= \sum_{n=0}^{\infty}\dfrac{q^{\binom{2n}{2}}}{[2n]_q!}z^{2n}\\
   \qSinh(z)&=& \dfrac{E_q(z)-E_q(-z)}{2}=\sum_{n=0}^{\infty}\dfrac{q^{\binom{2n+1}{2}}}{[2n+1]_q!}z^{2n+1}.
\end{eqnarray*}

%

\subsection{The $q$-Gamma functions}

The $q$-Gamma function of the first kind \cite{kac} is defined for $0<q<1$ as
\[\Gamma_q(t)=\int_{0}^{\infty}x^{t-1}E_q(-qx)d_qx,\quad t>0.\] 
It satisfies the fundemental relation 
\[\Gamma_q(t+1)=[t]_q\Gamma_q(t), \quad t>0.\]
Since for any nonnegative integer $n$ 
\[\Gamma_q(n+1)=[n]_q!,\]
it is clear that the $q$-Gamma function is a generalization of the $q$-factorial. 

The $q$-Gamma function of the second kind \cite{chung,diaz,sole} is definded by
\[\gamma_q(t)=\int_{0}^{\infty}x^{t-1}e_q(-x)d_qx,\;\quad (t>0)\]
and satisfied 
\[\gamma_q(1)=1,\quad \gamma_q(t+1)=q^{-t}[t]_q\gamma_q(t),\quad \gamma_q(n)=q^{-\binom{n}{2}}\Gamma_q(n),\quad (n\in\N).\]


\section{Double $q$-Laplace transform of the first kind}

Based on definitions \eqref{dlt} and \eqref{q1} we define the double $q$-Laplace transform of the first kind as
\begin{equation}\label{trans1}
\qlap{1}[f(x,y)](r,s)=\dqint f(x,y)E_q(-qrx)E_q(-qsy)d_qxd_qy, \;\; (r,s>0).
\end{equation}
Note that if $f(x,y)=g(x)h(y)$, then 
\begin{equation}\label{qlap-product1}
\qlap{1}[f(x,y)](r,s)=L_q\{g(x)\}(r)L_q\{h(y)\}(s).
\end{equation}
in particular, if $h(y)=1$, or $g(x)=1$, then \eqref{qlap-product1} reads
\begin{equation}\label{qlap-product101}
\qlap{1}[f(y)](r,s)=L_q\{1\}(r)L_q\{f(y)\}(s)=\dfrac{1}{r}L_q\{f(y)\}(s).
\end{equation}
and 
\begin{equation}\label{qlap-product102}
\qlap{1}[f(x)](r,s)=L_q\{g(x)\}(r)L_q\{1\}(s)=\frac{1}{s}L_q\{g(x)\}(r).
\end{equation}
\begin{proposition}
For any two complex numbers $\alpha$ and $\beta$, we have 
\[\qlap{1}\left\{\alpha f(x,y)+\beta g(x,y)\right\}=\alpha \mathcal{L}^{(1)}_{2,q}\{f(x,y)\}+\beta \mathcal{L}^{(1)}_{2,q}\{g(x,y)\}.\]
\end{proposition}

\begin{proof}
The proof follows from \eqref{trans1}. 
\end{proof}

\noindent In what follows, we give some examples. From \eqref{trans1}, we note that:

\begin{eqnarray*}
\qlap{1}\{1\}(r,s)&=&\dqint E_q(-qrx)E_q(-qsy)d_qxd_qy\\
 &=& \left(\qint E_q(-qrx)d_qx\right)\left(\qint E_q(-qsy)d_qy\right)\\
 &=& \dfrac{1}{r}\times\dfrac{1}{s}=\dfrac{1}{rs}.\\
 \qlap{1}\{xy\}(r,s)&=& \dqint xyE_q(-qrx)E_q(-qsy)d_qxd_qy\\
 &=& \left(\qint xE_q(-qrx)d_qx\right)\left(\qint yE_q(-qsy)d_qy\right)\\
 &=& \dfrac{1}{r^2}\times\dfrac{1}{s^2}=\dfrac{1}{(rs)^2}.
\end{eqnarray*} 
\[\qlap{1}\{1+4xy\}(r,s)=\qlap{1}\{1\}(r,s)+4\qlap{1}\{xy\}(r,s)=\dfrac{1}{rs}+\dfrac{4}{(rs)^2}.\]
We recall the following important relation \cite{kac}, 
\begin{equation}\label{int-change}
\qint f(\alpha x)d_qx=\dfrac{1}{\alpha}\qint f(x)d_qx,
\end{equation}
where $\alpha$ is a non zero complex number and $f$ is a one variable function. \\
Now we state the scaling theorem for $\qlap{1}$. 

\begin{theorem}
Let $a$ and $b$ two non zero complex numbers, $f$ a two variable function, then the following formula applies
\begin{equation}\label{scaling}
\qlap{1}\{f(ax,by)\}(r,s)=\dfrac{1}{ab}\qlap{1}\{f(x,y)\}\left(\frac{r}{a},\frac{s}{b}\right).
\end{equation}
\end{theorem}

\begin{proof}
Using relation \eqref{int-change}, we have 
\begin{eqnarray*}
\qlap{1}\{f(ax,by)\}(r,s)&=&\dqint f(ax,by)E_q(-qrx)E_q(-qsy)d_qxd_qy\\
&=& \qint \left(\qint f(ax,by) E_q(-qrx)d_qx\right) E_q(-qsy) d_qy\\
&=& \dfrac{1}{a}\qint \left(\qint f(x,by) E_q\left(-qx\frac{r}{a}\right)d_qx\right) E_q(-qsy) d_qy\\
&=& \dfrac{1}{a}\qint \left(\qint f(x,by)E_q(-qsy) d_qy \right) E_q\left(-qx\frac{r}{a}\right)d_qx\\
&=&\dfrac{1}{ab}\qint \left(\qint f(x,y)E_q\left(-qy\frac{s}{b}\right) d_qy \right) E_q\left(-qx\frac{r}{a}\right)d_qx\\
&=& \dfrac{1}{ab}\dqint f(x,y) E_q\left(-qx\frac{r}{a}\right)  E_q\left(-qy\frac{s}{b}\right)d_qxd_qy.
\end{eqnarray*}
and the proof of the Theorem is completed. 
\end{proof}
   
\begin{theorem}
For $\alpha>-1$, $\beta>-1$, we have the following 
\begin{equation}
 \qlap{1} \{x^\alpha y^\beta\}(r,s)=\dfrac{\Gamma_q(\alpha+1)}{r^{\alpha+1}}\dfrac{\Gamma_q(\beta+1)}{s^{\beta+1}}. 
\end{equation}
In particular, for $\alpha=n\in\N$ and $\beta=m\in\N$, we get 
\begin{equation}
\qlap{1}\{x^ny^m\}(r,s)=\dfrac{[n]_q![m]_q!}{r^{n+1}s^{m+1}}.
\end{equation}
\end{theorem}

\begin{proof}
The proof follows from the relation $L_q\{t^\alpha\}(s)=\dfrac{\Gamma_q(\alpha+1)}{s^{\alpha+1}}$ (see \cite{chung}) and the obvious equation 
\[\qlap{1} \{x^\alpha y^\beta\}(r,s)= L_q\{x^\alpha\}(r)\times L_q\{y^\beta\}(s). \]
\end{proof}

\noindent Let us take for example $\alpha =-\dfrac{1}{2}$ and $\beta=\dfrac{1}{2}$. Then we see that 
\begin{eqnarray*}
\qlap{1}\left(\sqrt{\frac{y}{x}}\right)(r,s)&=&L_q\{x^{-\frac{1}{2}}\}(r)\times L_q\{y^{\frac{1}{2}}\}(s)\\
&=&\Gamma_q\left(\frac{1}{2}\right)\Gamma_q\left(\frac{3}{2}\right)\dfrac{1}{s\sqrt{rs}},
\end{eqnarray*}
and for $\alpha =-\dfrac{1}{2}$ and $\beta=-\dfrac{1}{2}$ we have 
\begin{eqnarray*}
\qlap{1}\left(\frac{1}{\sqrt{xy}}\right)(r,s)&=&L_q\{x^{-\frac{1}{2}}\}(r)\times L_q\{y^{-\frac{1}{2}}\}(s)\\
&=&\left[\Gamma_q\left(\frac{1}{2}\right)\right]^2\dfrac{1}{\sqrt{rs}}.
\end{eqnarray*}

\begin{proposition}\label{prop-qadd}
Let $a$ and $b$ be two real numbers, then 
\begin{equation}\label{prop-qadd-eq}
\qlap{1}\big\{(ax\oplus_q by)^n\big\}(r,s)=\dfrac{[n]_q!}{br-as}\left(\left(\dfrac{b}{s}\right)^{n+1}-\left(\dfrac{a}{r}\right)^{n+1}\right).
\end{equation}
\end{proposition}

\begin{proof}
Combining the scaling property (see equation \eqref{scaling}) and \eqref{qaddition} we have 
\begin{eqnarray*}
\qlap{1}\big\{(ax\oplus_q by)^n\big\}(r,s)&=&\sum_{k=0}^{n}\qbinomial{n}{k}{q}\qlap{1}\left\{(ax)^{k}(by)^{n-k}\right\}(r,s)\\
&=& \dfrac{1}{ab}\sum_{k=0}^{n}\qbinomial{n}{k}{q}\qlap{1}\left\{x^{k}y^{n-k}\right\}\left(\frac{r}{a},\frac{s}{b}\right)\\
&=&\dfrac{1}{ab} \sum_{k=0}^{n}\qbinomial{n}{k}{q}\dfrac{[k]_q![n-k]_q!}{r^{k+1}s^{n-k+1}}a^{k+1}b^{n-k+1}\\
&=& \dfrac{[n]_q!}{ab}\left(\dfrac{a}{r}\right)\left(\dfrac{b}{s}\right)^{n+1}\sum_{k=0}^{n}\left(\dfrac{as}{rb}\right)^k\\
&=& \dfrac{[n]_q!}{br-as}\left(\left(\dfrac{b}{s}\right)^{n+1}-\left(\dfrac{a}{r}\right)^{n+1}\right).
\end{eqnarray*}
This ends the proof of the proposition.
\end{proof}

\begin{theorem}
Let $a$ and $b$ two complex numbers, then 
\begin{equation}\label{qexp01}
\qlap{1}\{e_q(ax\oplus_q by)\}(r,s)=\dfrac{1}{(r-a)(s-b)},\quad r>\Re(a),\;\; s>\Re(b).
\end{equation}
\end{theorem}

\begin{proof}
Using the definition of the $q$-addition \eqref{qadd}, and Proposition \ref{prop-qadd} we have 
\begin{eqnarray*}
\qlap{1}\{e_q(ax\oplus_q by)\}(r,s)&=& \sum_{n=0}^{\infty}\qlap{1}\left\{\dfrac{(ax\oplus_q by)^n}{[n]_q!}\right\}(r,s)\\
&=& \dfrac{1}{br-as}\sum_{n=0}^{\infty}\left(\left(\dfrac{b}{s}\right)^{n+1}-\left(\dfrac{a}{r}\right)^{n+1}\right)\\
&=& \dfrac{1}{br-as}\left(\dfrac{s}{s-b}-\dfrac{r}{r-a}\right)\\
&=& \dfrac{1}{(r-a)(s-b)}.
\end{eqnarray*}
\end{proof}
\noindent Note also that this result can be obtained using equations \eqref{qaddition}, \eqref{qlap-product1} and the fact that (see \cite{chung}):
\begin{equation}\label{qLap-qexp}
L_q(e_q(ax))(s)=\dfrac{1}{s-a}.
\end{equation}

\begin{proposition}
The following formulas apply
\begin{eqnarray}
\qlap{1}\{\qcos(ax\oplus_q by)\}(r,s)&=& \dfrac{rs-ab}{(r^2+a^2)(s^2+b^2)}\label{qcos-001}\\
\qlap{1}\{\qsin(ax\oplus_q by)\}(r,s)&=& \dfrac{as+br}{(r^2+a^2)(s^2+b^2)}.\label{qsin-002}
\end{eqnarray}
\end{proposition}

\begin{proof}
We indicate two proofs of these equations. First we can use  the relations (see \cite{kim})
\begin{eqnarray*}
  \qcos(x\oplus_q y)&=& \qcos(x)\qcos(y)-\qsin(x)\qsin(y),\\
  \qsin(x\oplus_q y)&=& \qsin(x)\qcos(y)+\qcos(x)\qsin(y),
\end{eqnarray*}
together with the equations \eqref{qlap-product1} and \eqref{qLap-qexp}.\\
For the second proof, we remark first that for any complex number $\lambda$, we have 
\[e_q(\lambda(x\oplus_q y))=e_q(\lambda x\oplus_q \lambda y),\] to write 
\begin{eqnarray*}
\qcos(ax\oplus_q by)&=&\dfrac{1}{2}\left(e_q(i(ax\oplus_q by))+e_q(-i(ax\oplus_q by))\right)\\
&=& \dfrac{1}{2}\left(e_q((aix\oplus_q biy))+e_q((-aix\oplus_q -biy))\right)\\
\qsin(ax\oplus_q by)&=&\dfrac{1}{2i}\left(e_q(i(ax\oplus_q by))-e_q(-i(ax\oplus_q by))\right)\\
&=& \dfrac{1}{2i}\left(e_q((aix\oplus_q biy))-e_q((-aix\oplus_q -biy))\right).
\end{eqnarray*}
Hence, using the linearity of $\qlap{1}$, and equation \eqref{qexp01}, it follows that
\begin{eqnarray*}
\qlap{1}\{\qcos(ax\oplus_q by)\}(r,s)&=& \dfrac{1}{2}\left\{\dfrac{1}{(r-ai)(s-bi)}+\dfrac{1}{(r+ai)(s+ib)}\right\}\\
&=& \dfrac{rs-ab}{(r^2+a^2)(s^2+b^2)}.
\end{eqnarray*} 
This proves again \eqref{qcos-001}. \eqref{qsin-002} follows in the same way.
\end{proof}

\begin{proposition}\label{qhyperbolic-addition}
The following equations apply
\begin{eqnarray*}
   \qcosh(x\oplus_q y)&=& \qcosh(x)\qcosh(y)+\qsinh(x)\qsinh(y),  \\
   \qsinh(x\oplus_q y)&=& \qcosh(x)\qsinh(y)+\qsinh(x)\qcosh(y).
\end{eqnarray*}
\end{proposition}

\begin{proof}
The proof uses the definitions of the involved functions. 
\end{proof}

\begin{proposition}
The following formulas apply
\begin{eqnarray}
\qlap{1}\{\qcosh(ax\oplus_q by)\}(r,s)&=& \dfrac{rs+ab}{(r^2-a^2)(s^2-b^2)}\label{qcosh-01}\\
\qlap{1}\{\qsinh(ax\oplus_q by)\}(r,s)&=& \dfrac{as+br}{(r^2-a^2)(s^2-b^2)}\label{qsinh-01}
\end{eqnarray}
\end{proposition}

\begin{proof}
The proof follows from Proposition \ref{qhyperbolic-addition}
and the equations \eqref{qlap-product1} and \eqref{qLap-qexp}. It can also be done using the fact that 
\begin{eqnarray*}
\qlap{1}\left\{\qcosh(ax\oplus_q by)\right\}(r,s)&=&\dfrac{1}{2}\qlap{1}\left\{\left(e_q(ax\oplus_q by)+e_q(-ax\oplus_q -by)\right)\right\}(r,s)\\
&=& \dfrac{1}{2}\left\{\dfrac{1}{(r-a)(s-b)}+\dfrac{1}{(r+a)(s+b)}\right\}\\
&=& \dfrac{rs+ab}{(r^2-a^2)(s^2-b^2)},
\end{eqnarray*}
which proves \eqref{qcosh-01}. \eqref{qsinh-01} can be obtained in a similar way.
\end{proof}

\begin{theorem} 
Let $f$ be a one variable function that has a $q$-Laplace transform. Assume that $f$ has the $q$-Taylor expansion 
\[f(x)=\sum_{n=0}^{\infty}a_n\dfrac{x^n}{[n]_q!},\] then the following relation holds:
{\small \begin{eqnarray}\label{qlap1-qadd}
   \qlap{1}\left[ f(\alpha x\oplus_q \beta y) \right](r,s)= \dfrac{1}{\alpha s-\beta r}\left({L}_q\big[f(x)\big]\left(\frac{r}{\alpha}\right)-{L}_q\big[f(x)\big]\left(\frac{s}{\beta}\right)\right).
 \end{eqnarray}}
\end{theorem}

\begin{proof}
We have the following 
\begin{eqnarray*}
  f(\alpha x\oplus_q \beta y) &=&  \sum_{n=0}^{\infty}a_n\dfrac{(\alpha x\oplus_q \beta y)^n}{[n]_q!} =\sum_{n=0}^{\infty}\left(\sum_{k=0}^{n}\qbinomial{n}{k}{q}(\alpha x)^k(\beta y)^{n-k}\right)\dfrac{a_n}{[n]_q!}.
\end{eqnarray*}
Hence it follows that 
\begin{eqnarray*}
  \qlap{1}\left[ f(x\oplus_q y) \right](r,s) &=&  \sum_{n=0}^{\infty}\left(\sum_{k=0}^{n}\qbinomial{n}{k}{q} \dfrac{\alpha^k[k]_q! \beta^{n-k}[n-k]_q!}{r^{k+1}s^{n+1-k}}\right)\dfrac{a_n}{[n]_q!}  \\
  &=&\sum_{n=0}^{\infty}\sum_{k=0}^{n}\dfrac{\alpha^k \beta^{n-k} a_n}{r^{k+1}s^{n+1-k}}\\
  &=& \dfrac{1}{\alpha s-\beta r}\left(\sum_{n=0}^{\infty}a_n\left(\dfrac{ \alpha}{r}\right)^{n+1}-\sum_{n=0}^{\infty}a_n\left(\dfrac{ \beta}{r}\right)^{n+1}\right)\\
  &=& \dfrac{1}{\alpha s-\beta r}\left({L}_q\big[f(x)\big]\left(\frac{r}{\alpha}\right)-{L}_q\big[f(x)\big]\left(\frac{s}{\beta}\right)\right).
\end{eqnarray*}
This ends the proof of the Theorem. 
\end{proof}


\noindent The next two theorems provide formulas for the double $q$-Laplace transform of the partial $q$-derivative and the partial $q$-derivatives of the double $q$-Laplace transform. Theses results are of great importance in the resolution of partial $q$-differential equations as we will see in section \ref{last-section}. 


\begin{theorem} 
The following equations hold true
{\small\begin{eqnarray}
 && \qlap{1}\left[\dfrac{\partial_q f}{\partial_q x}(x,y)\right](r,s)= r\qlap{1}\left[f(x,y)\right](r,s)- L_q\left[f(0,y)\right](s),\label{qder1}\\
 && \qlap{1}\left[\dfrac{\partial_q f}{\partial_q y}(x,y)\right](r,s)= s\qlap{1}\left[f(x,y)\right](r,s)- L_q\left[f(x,0)\right](r), \label{qder2}  \\
  &&\qlap{1}\left[\dfrac{\partial_q^2 f}{\partial_q x\partial_q y}(x,y)\right](r,s)=  rs\qlap{1}\left[f(x,y)\right](r,s)-r\qlap{1}\left[f(x,0)\right](r) \nonumber \\
  &&\hspace*{7cm}-s\qlap{1}\left[f(0,y)\right](s)+f(0,0), \label{qder3} \\
 && \qlap{1}\left[\dfrac{\partial_q^2 f}{\partial_q x^2}(x,y)\right](r,s)= r^2\qlap{1}\left[f(x,y)\right](r,s)-r\qlap{1}\left[f(0,y)\right](s)- L_q\left[\dfrac{\partial_q f}{\partial_q x} (0,y)\right](s), \nonumber  \\
  &&\qlap{1}\left[\dfrac{\partial_q^2 f}{\partial_q y^2}(x,y)\right](r,s)= s^2\qlap{1}\left[f(x,y)\right](r,s)-s\qlap{1}\left[f(x,0)\right](r)- L_q\left[\dfrac{\partial_q f}{\partial_q x} (x,0)\right](r).         \nonumber
\end{eqnarray}}
\end{theorem}

\begin{proof}
From definition \eqref{trans1}, and the formula of $q$-integration by parts, we have 
\begin{eqnarray*}
\qlap{1}\left[\dfrac{\partial_q f}{\partial_q x}(x,y)\right](r,s)&=& \dqint \dfrac{\partial_q f}{\partial_q x}(x,y)E_q(-qrx)E_q(-qsy)d_qxd_qy\\
&=& \qint\left(\qint \dfrac{\partial_q f}{\partial_q x}(x,y)E_q(-qrx)d_qx\right)E_q(-qsy)d_qy\\
&=& \qint\left( -f(0,y)+r\qint\!\! f(x,y)E_q(-qrx)d_qx\right)E_q(-qry)d_qy\\
&=& - L_q\left[f(0,y)\right] (s)+r\qlap{1}\left[f(x,y)\right](r,s).
\end{eqnarray*}
Hence \eqref{qder1} is proved. The proof of \eqref{qder3} uses \eqref{qder1}, \eqref{qder2} and the fact that (see \cite{kim}) 
\[L_q\left[\dfrac{\partial_q f}{\partial_q x} (x,0)\right](r)=rL_q\left[f(x,0)\right](r)-f(0,0).\]  The rest of the theorem in proved in the same way.
\end{proof}

The following theorem, which is obtained by induction from the previous one, is now stated without proof. 
\begin{theorem} [{\bf (Double Laplace transform of the Partial $q$-derivative)}]
The following equations are valid, where $n$ is a nonnegative integer. 
{\small \begin{eqnarray*}
  &&\qlap{1}\left[\dfrac{\partial_q^n f}{\partial_q x^n}(x,y)\right](r,s)= r^n\qlap{1}\left[f(x,y)\right](r,s)-\sum_{k=0}^{n-1}r^{n-1-k}L_q\left[\frac{\partial_q^kf}{\partial_q x^k}(0,y)\right](s),    \\          
  &&\qlap{1}\left[\dfrac{\partial_q^n f}{\partial_q y^n}(x,y)\right](r,s)= s^n\qlap{1}\left[f(x,y)\right](r,s)-\sum_{k=0}^{n-1}s^{n-1-k}L_q\left[\frac{\partial_q^kf}{\partial_q y^k}(x,0)\right](r).
\end{eqnarray*}}
\end{theorem}

\begin{remark}
Note that the expression 
\[L_q\left[\frac{\partial_q^nf}{\partial_q x^n}(0,y)\right](s)=s^n L_q\left[f(0,y)\right](s)-\sum_{k=0}^{n-1}s^{n-1-k}\frac{\partial_q^kf}{\partial_q x^k}(0,0). \]
is given in \cite{kim}.
\end{remark}

\begin{theorem}[{\bf (Partial $q$-derivative of the double Laplace transform)}]
The following relation is valid 
\begin{equation}
\qlap{1}\big[ x^my^n f(x,y)\big](r,s)=(-1)^{m+n}q^{\binom{m}{2}+\binom{n}{2}}\dfrac{\partial_q^{m+n}}{\partial_q s^n\partial_qr^m}\qlap{1}\big[f(x,y)\big]\left(q^{-m}r,q^{-n}s\right). 
\end{equation}
\end{theorem}

\begin{proof}
We recall the relation (see \cite[Theorem 2.4]{kim}) \[L_q\left[ x^n f(x)\right](s)=(-1)^nq^{\binom{n}{2}}\dfrac{\partial_q^n}{\partial_qs^n}L_q[f(x)]\big(q^{-n}s\big),\] from which we have:
{\small \begin{eqnarray*}
   &&\qlap{1}\big[ x^my^n f(x,y)\big](r,s)= \dqint x^my^nf(x,y)E_q(-rqx)E_q(-sqy)d_qxd_qy\\
   && =\int_{0}^{\infty}y^n \left( \int_{0}^{\infty}x^mf(x,y)E_q(-rqx)d_qx\right)E_q(-sqy)d_qy\\
   &&= \int_{0}^{\infty}y^n \left( (-1)^mq^{\binom{m}{2}}\dfrac{\partial_q^m}{\partial_qr^m}\int_{0}^{\infty}f(x,y)E_q(-q^{-m}rqx)d_qx\right) E_q(-sqy)d_qy\\  
   &&= (-1)^mq^{\binom{m}{2}}\dfrac{\partial_q^m}{\partial_qr^m}\int_{0}^{\infty} \left( (-1)^nq^{\binom{n}{2}}\dfrac{\partial_q^n}{\partial_qs^n}\int_{0}^{\infty}f(x,y)E_q(-q^{-n}sqy)d_qy\right) \\
   && \hspace*{8.5cm}E_q(-q^{-m}rqx)d_qx\\ 
   &&= { (-1)^{m+n}q^{\binom{m}{2}+\binom{n}{2}}\dfrac{\partial_q^{m+n}}{\partial_q s^n\partial_qr^m}\dqint f(x,y)E_q(-q^{-n}rqx)E_q(-q^{-n}sqy)d_qxd_qy}\\
   &&= (-1)^{m+n}q^{\binom{m}{2}+\binom{n}{2}}\dfrac{\partial_q^{m+n}}{\partial_q s^n\partial_qr^m}\qlap{1}\big[f(x,y)\big]\left(q^{-m}r,q^{-n}s\right).
\end{eqnarray*}}
This proves the theorem. 
\end{proof}

\section{Double $q$-Laplace transform of the second kind}

The double $q$-Laplace transform of the second kind is defined as 
\begin{equation}\label{trans2}
\qlap{2}[f(x,y)](r,s)=\dqint f(x,y)e_q(-rx)e_q(-sy)d_qxd_qy,\;\; (r,s>0).
\end{equation}

\noindent Note that if $f(x,y)=g(x)h(y)$, then 
\begin{equation}\label{qlap-product2}
\qlap{2}[f(x,y)](r,s)=\mathcal{L}_q\{g(x)\}(r)\mathcal{L}_q\{h(y)\}(s).
\end{equation}
In particular, if $h(y)=1$, or $g(x)=1$, then \eqref{qlap-product1} reads
\begin{equation}\label{qlap-product201}
\qlap{2}[f(y)](r,s)=\mathcal{L}_q\{1\}(r)\mathcal{L}_q\{f(y)\}(s)=\dfrac{1}{r}\mathcal{L}_q\{f(y)\}(s).
\end{equation}
and 
\begin{equation}\label{qlap-product202}
\qlap{2}[f(x)](r,s)=\mathcal{L}_q\{g(x)\}(r)\mathcal{L}_q\{1\}(s)=\frac{1}{s}\mathcal{L}_q\{g(x)\}(r).
\end{equation}
\begin{proposition}
For any two complex numbers $\alpha$ and $\beta$, we have 
\[\qlap{2}\left\{\alpha f(x,y)+\beta g(x,y)\right\}=\alpha \qlap{2}\{f(x,y)\}+\beta \qlap{2}\{g(x,y)\}.\]
\end{proposition}
\begin{proof}
The proof follows from \eqref{trans2}. 
\end{proof}

\begin{theorem}
Let $a$ and $b$ two non zero complex numbers, $f$ a two variable function, then the following formula applies
\begin{equation}\label{scaling2}
\qlap{2}\{f(ax,by)\}(r,s)=\dfrac{1}{ab}\qlap{2}\{f(x,y)\}\left(\frac{r}{a},\frac{s}{b}\right).
\end{equation}
\end{theorem}

\begin{proof}
Using relation \eqref{int-change}, we have 
\begin{eqnarray*}
\qlap{2}\{f(ax,by)\}(r,s)&=&\dqint f(ax,by)e_q(-rx)e_q(-sy)d_qxd_qy\\
&=& \qint \left(\qint f(ax,by) e_q(-rx)d_qx\right) e_q(-sy) d_qy\\
&=& \dfrac{1}{a}\qint \left(\qint f(x,by) e_q\left(-x\frac{r}{a}\right)d_qx\right) e_q(-sy) d_qy\\
&=& \dfrac{1}{a}\qint \left(\qint f(x,by)e_q(-sy) d_qy \right) e_q\left(-x\frac{r}{a}\right)d_qx\\
&=&\dfrac{1}{ab}\qint \left(\qint f(x,y)e_q\left(-y\frac{s}{b}\right) d_qy \right) e_q\left(-x\frac{r}{a}\right)d_qx\\
&=& \dfrac{1}{ab}\dqint f(x,y) e_q\left(-x\frac{r}{a}\right)  e_q\left(-y\frac{s}{b}\right)d_qxd_qy.
\end{eqnarray*}
and the proof of the Theorem is completed. 
\end{proof}

\begin{theorem}
For $\alpha>-1$, $\beta>-1$, we have the following 
\begin{equation}
 \qlap{2} \{x^\alpha y^\beta\}(r,s)=\dfrac{\gamma_q(\alpha+1)}{r^{\alpha+1}}\dfrac{\gamma_q(\beta+1)}{s^{\beta+1}}. 
\end{equation}
In particular, for $\alpha=n\in\N$ and $\beta=m\in\N$, we get 
\begin{equation}
\qlap{2}\{x^ny^m\}(r,s)=\dfrac{[n]_q!}{q^{\binom{n+1}{2}}r^{n+1}}\times \dfrac{[m]_q!}{q^{\binom{m+1}{2}}s^{m+1}}.
\end{equation}
\end{theorem}

\begin{proof}
The proof follows from the relation $\mathcal{L}_q\{t^\alpha\}(s)=\dfrac{\gamma_q(\alpha+1)}{s^{\alpha+1}}$ (see \cite{chung}) and the obvious equation 
\[\qlap{2} \{x^\alpha y^\beta\}(r,s)= \mathcal{L}_q\{x^\alpha\}(r)\times \mathcal{L}_q\{y^\beta\}(s). \]
\end{proof}

\begin{theorem}\label{theo-q2-qcoplus}
Let $a$ and $b$ be two complex numbers, then the following relation holds
\begin{eqnarray}
\qlap{2}\left\{(ax\boxplus_q by)^n\right\}(r,s)=\dfrac{q^{-\binom{n+1}{2}}[n]_q!}{br-as}\left(\left(\dfrac{b}{s}\right)^{n+1}-\left(\dfrac{a}{r}\right)^{n+1}\right).
\end{eqnarray}
\end{theorem}

\begin{proof}
From the definitions of the $q$-coaddition and the double $q$-Laplace transform of second kind, we have 
\begin{eqnarray*}
\qlap{2}\left\{(ax\boxplus_q by)^n\right\}(r,s)&=&\sum_{k=0}^{n}\qbinomial{n}{k}{q}q^{k(k-n)} \qlap{2}\left((ax)^k(by)^{n-k}\right)(r,s)\\
&=& \sum_{k=0}^{n}\qbinomial{n}{k}{q}q^{k(k-n)}\dfrac{a^k[k]_q!}{q^{\binom{k+1}{2}}r^{k+1}}\dfrac{b^{n-k}[n-k]_q!}{q^{\binom{n-k+1}{2}}s^{n-k+1}}\\
&=& \dfrac{q^{-\binom{n+1}{2}}[n]_q!b^n}{rs^{n+1}}\sum_{k=0}^{n}\left(\dfrac{as}{br}\right)^k\\
&=& \dfrac{q^{-\binom{n+1}{2}}[n]_q!b^n}{rs^{n+1}(br)^{n}}\dfrac{(br)^{n+1}-(as)^{n+1}}{br-as}\\
&=&\dfrac{q^{-\binom{n+1}{2}}[n]_q!}{br-as}\left(\left(\dfrac{b}{s}\right)^{n+1}-\left(\dfrac{a}{r}\right)^{n+1}\right).
\end{eqnarray*}
The theorem is then proved.
\end{proof}

\begin{theorem}
Let $a$ and $b$ be two complex numbers, then the following relation hold
\begin{eqnarray}
\qlap{2}\left\{E_q(ax\boxplus_q by)\right\}(r,s)=\dfrac{q^2}{(qr-a)(qs-b)},\;\; |r|>\left|\dfrac{a}{q}\right|,\; \; |s|>\left|\dfrac{b}{q}\right|. 
\end{eqnarray}
\end{theorem}

\begin{proof}
From Theorem \eqref{theo-q2-qcoplus} and the definition of the big $q$-exponential function, we have 
\begin{eqnarray*}
\qlap{2}\left\{E_q(ax\boxplus_q by)\right\}(r,s)&=&\sum_{n=0}^{\infty}\dfrac{q^{\binom{n}{2}}}{[n]_q!}\qlap{2}\{(ax\boxplus_q by)^n\}(r,s)\\
&=& \dfrac{1}{br-as}\left[\dfrac{b}{s}\sum_{n=0}^{\infty}\left(\dfrac{b}{qs}\right)^n- \dfrac{a}{r}\sum_{n=0}^{\infty}\left(\dfrac{a}{qr}\right)^n\right]\\
&=& \dfrac{1}{br-as}\left[\dfrac{b}{s}\dfrac{qs}{qs-b}-\dfrac{a}{r}\dfrac{qr}{qr-a} \right]\\
&=& \dfrac{q^2}{(qr-a)(qs-b)}. 
\end{eqnarray*}
Note that this result can be also proved using the fact that \[E_q(ax\boxplus_q by)=E_q(ax)E_q(by)\]
and the relation \cite{chung}
\[\mathcal{L}_q(E_q(ax))(r)=\dfrac{q}{qr-a}.\]
\end{proof}

\begin{proposition}
The following transforms hold 
\begin{eqnarray}
   \qlap{2}\left\{\qCos(ax\boxplus_q by)\right\}(r,s)&=&\dfrac{q^2(q^2rs-ab)}{((qr)^2+a^2)((qs)^2+b^2)}\label{f1}   \\
   \qlap{2}\left\{\qSin(ax\boxplus_q by)\right\}(r,s)&=&\dfrac{q^3(as+br)}{((qr)^2+a^2)((qs)^2+b^2)}\label{f2}   \\
   \qlap{2}\left\{\qCosh(ax\boxplus_q by)\right\}(r,s)&=&\dfrac{q^2(q^2rs+ab)}{((qr)^2-a^2)((qs)^2-b^2)}\label{f3}   \\
   \qlap{2}\left\{\qSinh(ax\boxplus_q by)\right\}(r,s)&=&\dfrac{q^3(as+br)}{((qr)^2-a^2)((qs)^2-b^2)}\label{f4}   
\end{eqnarray}
\end{proposition}

\begin{proof}
We have 
\begin{eqnarray*}
 \qlap{2}\left\{\qCos(ax\boxplus_q by)\right\}(r,s)&=&\dfrac{1}{2}\qlap{2}\left[E_q(iax\boxplus_q iby)+E_q(-iax\boxplus_q -iby)\right](r,s)\\
 &=& \dfrac{q^2}{(qr-ia)(qs-ib)}+\dfrac{q^2}{(qr+ia)(qs+ib)}\\
 &=& \dfrac{q^2(q^2rs-ab)}{((qr)^2+a^2)((qs)^2+b^2)}.
\end{eqnarray*}
So \eqref{f1} is proved. \eqref{f2}, \eqref{f3} and \eqref{f4} are proved in the same way. 
\end{proof}

\begin{theorem} 
Let $f$ be a one variable function that has a $q$-Laplace transform. Assume that $f$ has the $q$-Taylor expansion 
\[f(x)=\sum_{n=0}^{\infty}a_nq^{\binom{n}{2}}\dfrac{x^n}{[n]_q!},\] then the following relation holds:
{\small\begin{eqnarray}
   \qlap{2}\left[ f(\alpha x\boxplus_q \beta y) \right](r,s) = \dfrac{1}{\alpha s-\beta r}\left(\mathcal{L}_q\big[f(x)\big]\left(\frac{r}{\alpha}\right)-\mathcal{L}_q\big[f(x)\big]\left(\frac{s}{\beta}\right)\right).\label{qlap2-qadd}
\end{eqnarray}}
\end{theorem}

\begin{proof}
Assume that $f$ has the expansion as $f(x)=\displaystyle{\sum_{n=0}^{\infty}a_nq^{\binom{n}{2}}\dfrac{x^n}{[n]_q!}}$. Then, 
\begin{eqnarray*}
  \qlap{2}\left[ f(\alpha x\boxplus_q \beta y) \right](r,s) &=& \sum_{n=0}^{\infty}a_n\dfrac{q^{\binom{n}{2}}}{[n]_q!}\qlap{2}\left\{(\alpha x\boxplus_q \beta y)^n\right\}(r,s)\\
  &=& \sum_{n=0}^{\infty}a_n\dfrac{q^{\binom{n}{2}}}{[n]_q!}
  \dfrac{q^{-\binom{n+1}{2}}[n]_q!}{\beta r-\alpha s}\left(\left(\dfrac{\beta}{s}\right)^{n+1}-\left(\dfrac{\alpha}{r}\right)^{n+1}\right)\\
  &=& \dfrac{1}{\beta r-\alpha s}\left(\dfrac{\beta}{s}\sum_{n=0}^{\infty}a_n\left(\dfrac{\beta}{qs}\right)^{n}-\dfrac{\alpha}{r}\sum_{n=0}^{\infty}a_n\left(\dfrac{\alpha}{qr}\right)^{n}\right)\\
  &=& \dfrac{1}{\alpha s-\beta r}\left(\mathcal{L}_q\big[f(x)\big]\left(\frac{r}{\alpha}\right)-\mathcal{L}_q\big[f(x)\big]\left(\frac{s}{\beta}\right)\right).
\end{eqnarray*}
So the theorem is proved. 
\end{proof}
\textcolor{black}{
\begin{theorem}
The following equations hold true
{\small\begin{eqnarray}
  \qlap{2}\left[\dfrac{\partial_q f}{\partial_q x}(x,y)\right](r,s)&=& rq^{-1}\qlap{2}\left[f(x,y)\right](rq^{-1},s)- \mathcal{L}_q\left[f(0,y)\right](s),\label{2qder1}\\
  \qlap{2}\left[\dfrac{\partial_q f}{\partial_q y}(x,y)\right](r,s)&=& sq^{-1}\qlap{2}\left[f(x,y)\right](r,sq^{-1})- \mathcal{L}_q\left[f(x,0)\right](r), \label{2qder2}  \\
  \qlap{2}\left[\dfrac{\partial_q^2 f}{\partial_q x\partial_q y}(x,y)\right](r,s)&=&  rsq^{-2}\qlap{2}\left[f(x,y)\right](rq^{-1},sq^{-1})+f(0,0) \nonumber \\
  &&\hspace*{-.5cm}-rq^{-1}\mathcal{L}_q\left[f(x,0)\right](rq^{-1})-sq^{-1}\mathcal{L}_q\left[f(0,yq^{-1})\right](sq^{-1}), \label{2qder3} \\
  \qlap{2}\left[\dfrac{\partial_q^2 f}{\partial_q x^2}(x,y)\right](r,s)&=& r^2q^{-3}\qlap{2}\left[f(x,y)\right](rq^{-2},s)-rq^{-1}\qlap{1}\left[f(0,y)\right](s)\nonumber\\
  &&\hspace*{2.5cm}- sq^{-1}\mathcal{L}_q\left[f(0,y)\right](sq^{-1})+f(0,0),\nonumber \\
  \qlap{2}\left[\dfrac{\partial_q^2 f}{\partial_q y^2}(x,y)\right](r,s)&=&  s^2q^{-3}\qlap{2}\left[f(x,y)\right](r,sq^{-2})-rq^{-1}\qlap{1}\left[f(x,0)\right](rq^{-1})\nonumber\\
  &&\hspace*{2.5cm}- rq^{-1}\mathcal{L}_q\left[f(x,0)\right](r)+f(0,0),\phantom{aaaaaaaaaa}\nonumber
\end{eqnarray}}
\end{theorem}
\begin{proof}
From definition \eqref{trans1}, and the formula of $q$-integration by parts, we have 
{\small\begin{eqnarray*}
\qlap{2}\left[\dfrac{\partial_q f}{\partial_q x}(x,y)\right](r,s)&=& \dqint \dfrac{\partial_q f}{\partial_q x}(x,y)e_q(-rx)e_q(-sy)d_qxd_qy\\
&=& \qint\left(\qint \dfrac{\partial_q f}{\partial_q x}(x,y)e_q(-rx)d_qx\right)e_q(-sy)d_qy\\
&=& \qint\left( -f(0,y)+r\qint\!\! f(qx,y)e_q(-rx)d_qx\right)e_q(-sy)d_qy\\
&=& - \mathcal{L}_q\left[f(0,y)\right](s)+rq^{-1}\qlap{2}\left[f(x,y)\right](rq^{-1},s).
\end{eqnarray*}}
Hence \eqref{qder1} is proved. The proof of \eqref{qder3} uses \eqref{qder1}, \eqref{qder2} and the fact that (see \cite{kim}) 
\[\mathcal{L}_q\left[\dfrac{\partial_q f}{\partial_q x} (x,0)\right](r)=rq^{-1}\mathcal{L}_q\left[f(x,0)\right](rq^{-1})-f(0,0).\]  The rest of the theorem in proved in the same way.
\end{proof}}
\begin{theorem}[{\bf (Partial $q$-derivative of the double $q$-Laplace transform)}]
The following relation is valid 
\begin{equation}
\qlap{2}\big[ x^my^n f(x,y)\big](r,s)=(-1)^{m+n}\dfrac{\partial_q^{m+n}}{\partial_q s^n\partial_qr^m}\qlap{2}\big[f(x,y)\big]\left(r,s\right). 
\end{equation}
\end{theorem}
\begin{proof}
We recall the relation (see \cite[Theorem 3.5.]{kim}) \[\mathcal{L}_q\left[ x^n f(x)\right](s)=(-1)^n\dfrac{\partial_q^n}{\partial_qs^n}\mathcal{L}_q[f(x)]\big(s\big),\] from which we have:
\begin{eqnarray*}
   &&\qlap{2}\big[ x^my^n f(x,y)\big](r,s)= \dqint x^my^nf(x,y)e_q(-rx)e_q(-sy)d_qxd_qy\\
   && =\int_{0}^{\infty}y^n \left( \int_{0}^{\infty}x^mf(x,y)e_q(-rx)d_qx\right)e_q(-sy)d_qy\\
   && = \int_{0}^{\infty}y^n \left( (-1)^m\dfrac{\partial_q^m}{\partial_qr^m}\int_{0}^{\infty}f(x,y)e_q(-rx)d_qx\right) E_q(-sqy)d_qy\\ 
   &&= {\small (-1)^m\dfrac{\partial_q^m}{\partial_qr^m}\int_{0}^{\infty} \left( (-1)^n\dfrac{\partial_q^n}{\partial_qs^n}\int_{0}^{\infty}f(x,y)e_q(-sy)d_qy\right) e_q(-rx)d_qx}\\ 
   &&=  (-1)^{m+n}\dfrac{\partial_q^{m+n}}{\partial_q s^n\partial_qr^m}\dqint f(x,y)e_q(-rx)e_q(-sy)d_qxd_qy\\
   &&= (-1)^{m+n}\dfrac{\partial_q^{m+n}}{\partial_q s^n\partial_qr^m}\qlap{2}\big[f(x,y)\big]\left(r,s\right).
\end{eqnarray*}
This proves the theorem. 
\end{proof}

\section{Double $q$-Laplace transform of the third kind}

The double $q$-Laplace transform of the third kind is defined as 
\begin{equation}\label{trans3}
\qlap{3}[f(x,y)](r,s)=\dqint f(x,y)e_q(-rx)E_q(-qsy)d_qxd_qy, \;\; (r,s>0).
\end{equation}

\noindent Note that if $f(x,y)=g(x)h(y)$, then 
\begin{equation}\label{qlap-product3}
\qlap{3}[f(x,y)](r,s)=\mathcal{L}_q\{g(x)\}(r){L}_q\{h(y)\}(s).
\end{equation}
\begin{proposition}
For any two complex numbers $\alpha$ and $\beta$, we have 
\[\qlap{3}\left\{\alpha f(x,y)+\beta g(x,y)\right\}=\alpha \qlap{3}\{f(x,y)\}+\beta \qlap{3}\{g(x,y)\}.\]
\end{proposition}

\begin{proof}
The proof follows from \eqref{trans3}. 
\end{proof}

\begin{theorem}
Let $a$ and $b$ two non zero complex numbers, $f$ a two variable function, then the following formula applies
\begin{equation}\label{scaling3}
\qlap{3}\{f(ax,by)\}(r,s)=\dfrac{1}{ab}\qlap{3}\{f(x,y)\}\left(\frac{r}{a},\frac{s}{b}\right).
\end{equation}
\end{theorem}

\begin{proof}
Using relation \eqref{int-change}, we have 
\begin{eqnarray*}
\qlap{3}\{f(ax,by)\}(r,s)&=&\dqint f(ax,by)e_q(-rx)E_q(-qsy)d_qxd_qy\\
&=& \qint \left(\qint f(ax,by) e_q(-rx)d_qx\right) E_q(-qsy) d_qy\\
&=& \dfrac{1}{a}\qint \left(\qint f(x,by) e_q\left(-x\frac{r}{a}\right)d_qx\right) E_q(-qsy) d_qy\\
&=& \dfrac{1}{a}\qint \left(\qint f(x,by)E_q(-qsy) d_qy \right) e_q\left(-x\frac{r}{a}\right)d_qx\\
&=&\dfrac{1}{ab}\qint \left(\qint f(x,y)E_q\left(-qy\frac{s}{b}\right) d_qy \right) e_q\left(-x\frac{r}{a}\right)d_qx\\
&=& \dfrac{1}{ab}\dqint f(x,y) e_q\left(-x\frac{r}{a}\right)  E_q\left(-qy\frac{s}{b}\right)d_qxd_qy.
\end{eqnarray*}
and the proof of the Theorem is completed. 
\end{proof}

\begin{proposition}
For $\alpha>-1$, $\beta>-1$, we have the following 
\begin{equation}
 \qlap{3} \{x^\alpha y^\beta\}(r,s)=\dfrac{\gamma_q(\alpha+1)}{r^{\alpha+1}}\dfrac{\Gamma_q(\beta+1)}{s^{\beta+1}}. 
\end{equation}
In particular, for $\alpha=n\in\N$ and $\beta=m\in\N$, we get 
\begin{equation}
\qlap{3}\{x^ny^m\}(r,s)=\dfrac{[n]_q!}{q^{\binom{n+1}{2}}r^{n+1}}\times \dfrac{[m]_q!}{s^{m+1}}.
\end{equation}
\end{proposition}

\begin{proof}
The proof follows from relations $\mathcal{L}_q\{t^\alpha\}(s)=\dfrac{\gamma_q(\alpha+1)}{s^{\alpha+1}}$ and ${L}_q\{t^\alpha\}(s)=\dfrac{\Gamma_q(\alpha+1)}{s^{\alpha+1}}$ (see \cite{chung}) and the obvious equation (from \eqref{qlap-product3})
\[\qlap{3} \{x^\alpha y^\beta\}(r,s)= \mathcal{L}_q\{x^\alpha\}(r)\times {L}_q\{y^\beta\}(s). \]
\end{proof}

\textcolor{black}{
\begin{theorem}
The following equation applies 
\begin{eqnarray}
\qlap{3}\{(ax\oplus by)_{q}^n\}(r,s)=\dfrac{[n]_q!}{q^{\binom{n+1}{2}}(as-brq^n)}\left(\left(\frac{a}{s}\right)^{n+1}- \left(\frac{bq^n}{r}\right)^{n+1}  \right).
\end{eqnarray}
\end{theorem}
\begin{proof}
We have 
\begin{eqnarray*}
\qlap{3}\{(ax\oplus by)_{q}^n\}(r,s)&=& \sum_{k=0}^n\qbinomial{n}{k}{q}q^{\binom{k}{2}}a^{n-k}b^{k}\qlap{3}\{x^{n-k}y^{k}\}(r,s)\\
&=& \sum_{k=0}^n\qbinomial{n}{k}{q}q^{\binom{k}{2}}a^{n-k}b^{k}\dfrac{[n-k]_q!}{q^{\binom{n-k+1}{2}}r^{n-k+1}}\times \dfrac{[k]_q!}{s^{k+1}}\\
\end{eqnarray*}
\begin{eqnarray*}
\phantom{\qlap{3}\{}&=&\dfrac{[n]_q!a^n}{q^{\binom{n+1}{2}}sr^{n+1}}\sum_{k=0}^n\left(\frac{q^nbr}{as}\right)^k=\dfrac{[n]_q!a^n}{q^{\binom{n+1}{2}}sr^{n+1}}\times \dfrac{1-\left(\frac{q^nbr}{as}\right)^{n+1}}{1-\frac{q^nbr}{as}}\\
&=&\dfrac{[n]_q!}{q^{\binom{n+1}{2}}(as-brq^n)}\left(\left(\frac{a}{s}\right)^{n+1}- \left(\frac{bq^n}{r}\right)^{n+1}  \right)   
\end{eqnarray*}
\end{proof}}

\begin{proposition}
The following equation applies
\begin{equation*}
\qlap{3}\left\{E_q(ax)e_q(by)\right\}(r,s)=\dfrac{q}{(qr-a)(s-b)}.
\end{equation*}
\end{proposition}

\section{Double $q$-Laplace transform of the fourth kind}

The double $q$-Laplace transform of the third kind by the following 
\begin{equation}
\qlap{4}[f(x,y)](r,s)=\dqint f(x,y)E_q(-qrx)e_q(-sy)d_qxd_qy, \;\; (r,s>0).
\end{equation}
We give without prove some important properties of the double $q$-Laplace transform of the fourth kind. These results can be obtained easily as those of the double $q$-Laplace transform of the third kind.

\noindent Note that if $f(x,y)=g(x)h(y)$, then 
\begin{equation}\label{qlap-product4}
\qlap{4}[f(x,y)](r,s)={L}_q\{g(x)\}(s)\mathcal{L}_q\{h(y)\}(r).
\end{equation}
\begin{proposition}
For any two complex numbers $\alpha$ and $\beta$, we have 
\[\qlap{4}\left\{\alpha f(x,y)+\beta g(x,y)\right\}=\alpha \qlap{4}\{f(x,y)\}+\beta \qlap{4}\{g(x,y)\}.\]
\end{proposition}

\begin{proposition}
Let $a$ and $b$ two non zero complex numbers, $f$ a two variable function, then the following formula applies
\begin{equation}\label{scaling3}
\qlap{4}\{f(ax,by)\}(r,s)=\dfrac{1}{ab}\qlap{4}\{f(x,y)\}\left(\frac{r}{a},\frac{s}{b}\right).
\end{equation}
\end{proposition}

\begin{proposition}
For $\alpha>-1$, $\beta>-1$, we have the following 
\begin{equation}
 \qlap{4} \{x^\alpha y^\beta\}(r,s)=\dfrac{\Gamma_q(\alpha+1)}{r^{\alpha+1}}\dfrac{\gamma_q(\beta+1)}{s^{\beta+1}}. 
\end{equation}
In particular, for $\alpha=n\in\N$ and $\beta=m\in\N$, we get 
\begin{equation}
\qlap{4}\{x^ny^m\}(r,s)=\dfrac{[n]_q!}{r^{n+1}}\times \dfrac{[m]_q!}{q^{\binom{m+1}{2}}s^{m+1}}.
\end{equation}
\end{proposition}

\begin{proposition}
The following relation holds 
\begin{equation}
\qlap{4}\{(ax\oplus by)_q^n\}(r,s)=\dfrac{[n]_q!}{br-qsa}\left(\left(\frac{b}{s}\right)^{n+1}-\left(\frac{qa}{r}\right)^{n+1}\right).
\end{equation}
\end{proposition}

\begin{proposition}
The following equation applies 
\begin{equation}
\qlap{4}\{e_q(ax)E_q(by)\}(r,s)=\dfrac{q}{(r-a)(qs-b)}.
\end{equation}
\end{proposition}

\section{Some applications} \label{last-section}

\subsection{Application to some $q$-functional equations}


\subsubsection{The first $q$-Cauchy's functional equation}

\noindent We consider the following $q$-Cauchy's functional equation 
\begin{equation}\label{cauchy1}
f(x\oplus_q y)=f(x)+f(y),
\end{equation}
where $f$ is an unknown function. \\
We apply the double $q$-Laplace transform $\qlap{1}$ to \eqref{cauchy1} combined with  \eqref{qlap1-qadd}, \eqref{qlap-product101} and \eqref{qlap-product102}, to get 
\begin{equation*}
\dfrac{1}{s-r}\left[ L_q[f(x)](r)-L_q[f(y)](s)\right]=\dfrac{1}{s}L_q[f(x)](r)+\dfrac{1}{r}L_q[f(y)](s)
\end{equation*}
that is 
\[L_q[f(x)](r)\left[\dfrac{1}{s-r}-\dfrac{1}{s}\right]=L_q[f(y)](s)\left[\dfrac{1}{s-r}+\dfrac{1}{r}\right].\]
Simplifying this equation, we obtain 
\[r^2L_q[f(x)](r)=q^2L_q[f(y)](s),\]
where the left hand side is a function of $r$ alone and the right hand side is a function of $s$ alone. This equation is true provided each side is equal to an arbitrary constant $k$ so that 
\[r^2L_q[f(x)](r)=k,\]
or 
\[L_q[f(x)](r)=\dfrac{k}{r^2}.\]
The inverse transform gives the solution of the $q$-Cauchy functional equation \eqref{cauchy1} as 
\begin{equation}
f(x)=kx,
\end{equation}
where $k$ is an arbritrary constant. 

\subsubsection{The second $q$-Cauchy's functional equation}

\noindent We consider the following $q$-Cauchy's functional equation 
\begin{equation}\label{cauchy2}
f(x\boxplus_q y)=f(x)+f(y),
\end{equation}
where $f$ is an unknown function.\\
We apply the double $q$-Laplace transform $\qlap{2}$ to \eqref{cauchy2} combined with  \eqref{qlap2-qadd}, \eqref{qlap-product201} and \eqref{qlap-product202}, to get 
\begin{equation*} 
\dfrac{1}{s-r}\left[\mathcal{L}_q[f(x)](r)-\mathcal{L}_q[f(y)](s)\right]=\dfrac{1}{s}\mathcal{L}_q[f(x)](r)+\dfrac{1}{r}\mathcal{L}_q[f(y)](s)
\end{equation*}
that is 
\[\mathcal{L}_q[f(x)](r)\left[\dfrac{1}{s-r}-\dfrac{1}{s}\right]=\mathcal{L}_q[f(y)](s)\left[\dfrac{1}{s-r}+\dfrac{1}{r}\right].\]
Simplifying this equation, we obtain 
\[r^2\mathcal{L}_q[f(x)](r)=q^2\mathcal{L}_q[f(y)](s),\]
where the left hand side is a function of $r$ alone and the right hand side is a function of $s$ alone. This equation is true provided each side is equal to an arbitrary constant $k$ so that 
\[r^2\mathcal{L}_q[f(x)](r)=k,\]
or 
\[\mathcal{L}_q[f(x)](r)=\dfrac{k}{r^2}.\]
The inverse transform gives the solution of the $q$-Cauchy functional equation \eqref{cauchy2} as 
\begin{equation}
f(x)=kqx,
\end{equation}
where $k$ is an arbritrary constant. 

\subsubsection{The first $q$-Cauchy-Abel's functional equation}

\noindent We consider the following $q$-Cauchy-Abel's functional equation 
\begin{equation}\label{abel1}
f(x\oplus_q y)=f(x)f(y),
\end{equation}
where $f$ is an unknown function.\\
We apply the double $q$-Laplace transform $\qlap{1}$ to \eqref{abel1} combined with  \eqref{qlap1-qadd} and  \eqref{qlap-product1} to get 
\begin{equation*}
\dfrac{1}{s-r}\left[ L_q[f(x)](r)-L_q[f(y)](s)\right]=L_q[f(x)](r)L_q[f(y)](s)
\end{equation*}
that is 
\[\dfrac{1-rL_q[f(x)](r)}{L_q[f(x)](r)}=\dfrac{1-sL_q[f(y)](s)}{L_q[f(y)](s)},\]
where the left hand side is a function of $r$ alone and the right hand side is a function of $s$ alone. This equation is true provided each side is equal to an arbitrary constant $k$ so that 
\[\dfrac{1-rL_q[f(x)](r)}{L_q[f(x)](r)}=k,\]
or 
\[L_q[f(x)](r)=\dfrac{1}{r+k}.\]
The inverse transform gives the solution of the $q$-Cauchy-Abel's functional equation \eqref{abel1} as 
\begin{equation}
f(x)=e_q(-kx),
\end{equation}
where $k$ is an arbritrary constant.

\subsubsection{The second $q$-Cauchy-Abel's functional equation}

\noindent We consider the following $q$-Cauchy-Abel's functional equation 
\begin{equation}\label{abel2}
f(x\boxplus_q y)=f(x)f(y),
\end{equation}
where $f$ is an unknown function. \\
We apply the double $q$-Laplace transform $\qlap{2}$ to \eqref{abel2} combined with  \eqref{qlap2-qadd} and  \eqref{qlap-product2} to get 
\begin{equation*}
\dfrac{1}{s-r}\left[ \mathcal{L}_q[f(x)](r)-\mathcal{L}_q[f(y)](s)\right]=\mathcal{L}_q[f(x)](r)\mathcal{L}_q[f(y)](s)
\end{equation*}
that is 
\[\dfrac{1-r\mathcal{L}_q[f(x)](r)}{\mathcal{L}_q[f(x)](r)}=\dfrac{1-s\mathcal{L}_q[f(y)](s)}{\mathcal{L}_q[f(y)](s)},\]
where the left hand side is a function of $r$ alone and the right hand side is a function of $s$ alone. This equation is true provided each side is equal to an arbitrary constant $k$ so that 
\[\dfrac{1-r\mathcal{L}_q[f(x)](r)}{\mathcal{L}_q[f(x)](r)}=k,\]
or 
\[\mathcal{L}_q[f(x)](r)=\dfrac{1}{r+k}=\dfrac{q}{qr+qk}.\]
The inverse transform gives the solution of the $q$-Cauchy-Abel's functional equation \eqref{abel2} as 
\begin{equation}
f(x)=E_q(-qkx),
\end{equation}
where $k$ is an arbritrary constant.

\subsection{Application to some partial $q$-differential equations}

\subsubsection{The $q$-transport equation}

\noindent We introduce the following $q$-tansport equation 
\begin{equation}\label{qtransport}
  \frac{\partial_q u}{\partial_q t}(x,t)+c\frac{\partial_q u}{\partial_q x}(x,t)=0,
\end{equation}
with 
\begin{equation}\label{qtransport-boarding-conditions}
u(x,0)=f(x),\;\; x>0\;\;\; \textrm{and}\;\;\; u(0,t)=g(t),\;\; t>0. 
\end{equation}
Applying the double $q$-Laplace transform $\qlap{1}$ to \eqref{qtransport} combinded with \eqref{qder1}, \eqref{qder2} and \eqref{qtransport-boarding-conditions}, we get 
\[s\qlap{1}[u(x,t)](r,s)-L_q[f(x)](r)+c\left[r\qlap{1}[u(x,t)](r,s)-L_q[g(t)](s) \right]=0\]
that is 
\[\qlap{1}[u(x,t)](r,s)=\dfrac{cL_q[g(t)](s)+L_q[f(x)](r)}{s+cr}.\]
Hence, 
\begin{equation}
u(x,t)=\left(\qlap{1}\right)^{-1}\left[ \dfrac{cL_q[g(t)](s)+L_q[f(x)](r)}{s+cr} \right].
\end{equation}
In particular,
\begin{itemize}
   \item if $u(x,0)=f(x)=1$ and $u(0,t)=g(t)=1$, then 
   \begin{eqnarray*}
   u(x,t)&=&  \left(\qlap{1}\right)^{-1}\left[ \dfrac{cL_q[g(t)](s)+L_q[f(x)](r)}{s+cr} \right](x,t)\\
   &=& \left(\qlap{1}\right)^{-1}\left[\dfrac{c/s+1/r}{s+cr}\right](x,t)=\left(\qlap{1}\right)^{-1}\left[\dfrac{1}{sr}\right](x,t)=1.
   \end{eqnarray*}
   \item if $c=-1$,  $u(x,0)=f(x)=x^n$ and $u(0,t)=g(t)=t^n$ with $n\in\N$, then 
   \begin{eqnarray*}
   u(x,t)&=&  \left(\qlap{1}\right)^{-1}\left[ \dfrac{-L_q[t^n](s)+L_q[x^n](r)}{s-r} \right](x,t)\\
   &=& \left(\qlap{1}\right)^{-1}\left[\dfrac{-[n]_q!/s^{n+1}+[n]_q!/r^{n+1}}{s-r}\right](x,t)=(x\oplus_q t)^n.
   \end{eqnarray*}
   where \eqref{prop-qadd-eq} has been used. 
\end{itemize}

\subsubsection{The non-homogenous space-time $q$-telegraph equation}

\noindent We consider the non-homogenous space-time $q$-telegraph equation 
\begin{equation}\label{telegraph}
 c^2\dfrac{\partial^2_q u}{\partial_q x^2}(x,t)-\dfrac{\partial_q^2 u}{\partial_q t^2}(x,t)-(\alpha+\beta)\dfrac{\partial_q u}{\partial_q t}(x,t)-\alpha\beta u(x,t)=[c^2-(\alpha+1)(\beta+1)]e_q(x\oplus_q t),
\end{equation}
with the conditions 
\begin{eqnarray*}
   u(0,t)&=&e_q(t)\\
   u(x,0)&=&e_q(x)\\
   \dfrac{\partial_q u}{\partial_q x}(0,t)&=&e_q(t)\\
   \dfrac{\partial_q u}{\partial_q t}(x,0)&=&e_q(x).
\end{eqnarray*}
Applying $\qlap{1}$ to \eqref{telegraph}, we obtain 
\begin{eqnarray*}
&&c^2\left\{r^2\qlap{1}[u(x,t)](r,s)-rL_q[u(0,t)](s)-L_q\left[ \dfrac{\partial_q u}{\partial_q x}(0,t)\right](s)\right\}\\
&&\hspace*{1cm} -\left\{s^2\qlap{1}[u(x,t)](r,s)-sL_q[u(x,0)](r)-L_q\left[ \dfrac{\partial_q u}{\partial_q x}(x,0)\right](r)\right\}\\
&&\hspace*{1cm}-(\alpha+\beta)\left\{ s\qlap{1}[u(x,t)](r,s)-L_q[u(x,0)](r) \right\}-\alpha\beta\qlap{1}[u(x,t)](r,s)\\
&&\hspace*{4cm} =[c^2-(\alpha+1)(\beta+1)]\qlap{1}[e_q(x\oplus_q t)](r,s),
\end{eqnarray*}
Using the conditions and simplifying the result we obtain 
\[\qlap{1}[u(x,t)](r,s)=\dfrac{1}{(r-1)(s-1)},\]
and hence we have 
\[u(x,t)=e_q(x\oplus_q t).\]

%

%

\subsubsection{The $q$-wave equation}

\noindent We consider the following $q$-wave equation in a quarter plane 
\begin{equation}
\frac{\partial_q^2 u}{\partial_q t^2}(x,t)-c^2 \frac{\partial_q^2u}{\partial_q x^2}(x,t)=0, 
\end{equation}
with the initial contidion 
\begin{equation}
u(x,0)=f(x)\quad \textrm{and}\quad \dfrac{\partial_q u}{\partial_q t}(x,0)=g(x), x>0, 
\end{equation}
\begin{equation}
u(0,t)=0,\quad \textrm{and} \quad \dfrac{\partial_q u}{\partial_q x}(0,t)=0.
\end{equation}
We apply the double $q$-Laplace transform $\qlap{1}$ to have 
\begin{eqnarray*}
&&s^2\qlap{1}\left[u(x,t)\right](r,s)-sL_q[u(x,0)](r)-L_q\left[\dfrac{\partial_q u}{\partial_q t}(x,0)\right](r)\\
&&\hspace*{1cm}c^2\left\{r^2 \qlap{1}\left[u(x,t)\right](r,s)-rL_q[u(0,t)](s)-  L_q\left[\dfrac{\partial_q u}{\partial_q x}(0,t)\right](s)      \right\}=0.
\end{eqnarray*}
That is 
\[\qlap{1}\left[u(x,t)\right](r,s)=\dfrac{sL_q[f(x)](r)+L_q[g(x)](r)}{s^2-c^2r^2}.\]
Hence 
\[u(x,t)=\left(\qlap{1}\right)^{-1}\left[ \dfrac{sL_q[f(x)](r)+L_q[g(x)](r)}{s^2-c^2r^2}  \right](x,t).\]

\begin{remark}
Note that in \cite{brahim}, another $q$-wave equation is given combining the $q$-derivative with respect to $t$ and the classical derivative with respect to $x$ as 
\begin{equation*}
\frac{\partial_q^2 u}{\partial_q y^2}(x,y)-\frac{\partial^2u}{\partial x^2}(x,t)=0.
\end{equation*}
\end{remark}





\end{document}